\documentclass[a4paper,11pt]{amsart}
\usepackage[english]{babel}
\usepackage{hyperref}

\newtheorem{theorem}{Theorem}[section]
\newtheorem{lm}[theorem]{Lemma}
\newtheorem{corollary}[theorem]{Corollary}
\newtheorem{prop}[theorem]{Proposition}
\theoremstyle{definition}

\theoremstyle{remark}
\newtheorem{remark}[theorem]{Remark}

\usepackage{amssymb}

\usepackage{epsfig}
\usepackage{subfigure}

\def\SSN{\SS^{N-1}}

\newcommand{\NN}{\mathbb{N}}
\newcommand{\RR}{\mathbb{R}}

\newcommand{\lam}{\lambda}

\newcommand{\ep}{\varepsilon}

\newcommand{\de}{\delta}

\newcommand{\al}{\alpha}
\newcommand{\si}{\sigma}
\newcommand{\ga}{\gamma}

\newcommand{\Om}{\Omega}
\newcommand{\pa}{\partial}
\newcommand{\RN}{\RR^N}
\newcommand{\ovr}{\overline}

\newcommand{\De}{\Delta}

\def\NN{\mathbb N}
\def\SS{\mathbb S}

\def\cH{\mathcal{H}}

\def\cC{\mathcal{C}}

\def\RR{\mathbb R}
\def\SS{{\mathbb S}}
\def\RN{{\mathbb R^N}}

\def\Om{\Omega}
\def\De{\Delta}

\def\Si{\Sigma}
\def\al{\alpha}
\def\be{\beta}
\def\ga{\gamma}
\def\la{\lambda}
\def\La{\Lambda}

\def\om{\omega}
\def\pa{\partial}

\def\de{\delta}
\def\th{\theta}

\newcommand{\zi}{\zeta}

\def\ds{\displaystyle}
\def\ovr{\overline}

\newcommand{\pusn}{\big[u_\nu \big]_{\pa \Om}}

\DeclareMathOperator{\dist}{\mathrm{dist}}

\begin{document}
    \title[]{H\"{o}lder stability for \\ Serrin's 
overdetermined problem}

  %\today

\author[G. Ciraolo]{Giulio Ciraolo}
\address{Dipartimento di Matematica e
Informatica, Universit\`a di Palermo, Via Archirafi 34, 90123 Palermo, Italy.}
\email{giulio.ciraolo@unipa.it}
\urladdr{http://www.math.unipa.it/~g.ciraolo/}

\author[R. Magnanini]{Rolando Magnanini}
    \address{Dipartimento di Matematica ed Informatica ``U.~Dini'',
Universit\` a di Firenze, viale Morgagni 67/A, 50134 Firenze, Italy.}
    %\curraddr{...}
    \email{magnanin@math.unifi.it}
    \urladdr{http://web.math.unifi.it/users/magnanin}

%\author[S. Sakaguchi]{Shigeru Sakaguchi}
%    \address{Research Center for Pure and Applied Mathematics, Graduate School of Information Sciences,
%Tohoku University, Sendai 980-8579, Japan.}
%%\curraddr{...}
%    \email{sigersak@m.tohoku.ac.jp}
%\urladdr{http://researchmap.jp/sigersak2012415/}

\author[V. Vespri]{Vincenzo Vespri}
    \address{Dipartimento di Matematica ed Informatica ``U.~Dini'',
Universit\` a di Firenze, viale Morgagni 67/A, 50134 Firenze, Italy.}
    %\curraddr{...}
    \email{vespri@math.unifi.it}
    \urladdr{http://web.math.unifi.it/users/vespri}

%\dedicatory{...}
   % \date{October 5, 2006}
    %\thanks{...}
    %\translator{...}
\begin{abstract}
In a bounded domain $\Om$, we consider a positive solution of  the problem $\De u+f(u)=0$ in $\Om$,
$u=0$ on $\pa\Om$, where $f:\RR\to\RR$ is a locally Lipschitz continuous function. Under sufficient conditions on $\Om$ (for instance, if $\Om$ is convex), we show that $\pa\Om$
is contained in a spherical annulus of radii $r_i<r_e$, where $r_e-r_i\le C\,[u_\nu]_{\pa\Om}^\tau$ for some constants $C>0$ and $\tau\in (0,1]$. Here, $[u_\nu]_{\pa\Om}$ is the Lipschitz seminorm on $\pa\Om$ of the normal derivative of $u$. This result improves to H\"older stability the logarithmic estimate obtained in \cite{ABR} for Serrin's overdetermined problem. It also extends to a large class of semilinear equations the H\"older estimate obtained in \cite{BNST2} for the case of torsional rigidity ($f\equiv 1$) by means of integral identities. The proof hinges on ideas contained in \cite{ABR} and uses 
Carleson-type estimates and improved Harnack inequalities in cones.
\end{abstract}

%\today

\maketitle

\section[Introduction]{Introduction}
Serrin's overdetermined problem has been the object of many investigations.
In its classical form, it involves a sufficiently smooth bounded domain $\Om$ in $\RN$ and a
classical solution of the set of equations:
\begin{eqnarray}
&\Delta u + f(u) = 0 \ \mbox{ and } \
u\ge 0 \ \mbox{ in } \ \Om, \ u=0 \ \mbox{ on } \ \pa \Om,
\label{serrin1}
\\
&u_\nu=c \ \mbox{ on } \ \pa \Om.
\label{serrin2}
\end{eqnarray}
Here, $f:[0,+\infty)\to\RR$ is a locally Lipschitz continuous function, $u_\nu$ denotes the {\it inward} normal derivative of $u$ and $c$ is a positive constant. Under these assumptions, Serrin \cite{Se} proved that $\Om$ must be a ball and $u$ be radially symmetric. For his proof, he adapted and improved a method created by Aleksandrov to prove his {\it Soap Bubble Theorem} (see \cite{Al}). 
In fact, the radial symmetry of $\Om$ and $u$ is obtained by the so-called {\it method of moving planes}
based upon the observation that the euclidean ball is the only bounded domain that is symmetric with respect to any hyperplane passing through its center of mass. There are other interesting proofs of this symmetry result, based on integral identities, that generally need severe restrictions on $f$ (see for instance \cite{We}, \cite{PS}, \cite{BNST}).
\par
Overdetermined problems like \eqref{serrin1}-\eqref{serrin2} arise in many physical and geometric situations; they can be seen as a prototype of inverse problems and often emerge in free boundary and shape optimization problems (see \cite{He}). 
\par
Despite the intense research that has been devoted to them for almost five decades, there are still many open problems. An important one  --- the focus of this paper --- concerns the study of the stability of the radial configuration.  
\par
The first contribution
in this direction is \cite{ABR}, where
it is proved that if one assumes that $u_\nu$ is \emph{almost} constant on $\pa \Om$, then there exist two concentric balls $B_{r_i}$ and $B_{r_e}$, with 
\begin{equation}
\label{concentric balls}
B_{r_i} \subset \Om \subset B_{r_e}
\end{equation} 
and such that 
$r_e-r_i$ can be bounded in terms of some measure of the deviation of $u_\nu$ from being a constant. More precisely, in \cite{ABR} it is proved the estimate
\begin{equation} 
\label{stability ABR}
r_e-r_i\le C\,|\log  \left\|u_\nu-d\|_{C^1(\pa \Om)} \right|^{-1/N},
\end{equation}
where $d$ is some given constant, provided $\| u_\nu-d\|_{C^1(\pa \Om)}$ is sufficiently small; here, $C$ is a constant depending on $N$, the regularity of $\pa\Om$, the diameter of $\Om$ and the Lipschtz constant of $f$.
The proof is based on a quantitative study of the method of moving planes and works for a general locally Lipschitz non-linearity $f$.
\par
In the case of the {\it torsional rigidity} problem, that is when $f\equiv 1$, \eqref{stability ABR} was improved in \cite{BNST2} (see also \cite{BNST3} for Monge-Amp\`ere equations). Indeed, the authors replace the logarithmic dependence at the right-hand side of \eqref{stability ABR} 
by a power law of H\"older type. Furthermore, they also give a stability estimate in terms of the $L^1$-norm of the deviation instead of its $C^1$-norm.
\par
In our main result, Theorem \ref{stability serrin}, we will show that the logarithmic estimate
\eqref{stability ABR} can be improved to obtain a stability of H\"older type. In order to avoid unnecessary technicalities, in this section we present our result in a relevant particular case. In what follows, we denote by $d_\Om$ the {\it diameter} of $\Om$, $r_\Omega$ the radius of the optimal interior touching sphere to $\pa\Om$ (see Section \ref{sec:MovPlanes} for its precise definition), and use the following notation:
\begin{equation}
\label{seminorm normal derivative}
\pusn = \sup_{\substack{x,y \in \pa \Om\\ \ x \neq y}} \frac{|u_\nu(x) - u_\nu(y)|}{|x-y|}.
\end{equation}

\begin{theorem}
\label{th:stability-convex}
Let $\Om \subset \RR^N$ be a convex domain with boundary of class $C^{2,\al}$
and let $f$ be a locally Lipschitz continuous function such that $f(0)\ge 0$.  
Let $u\in C^{2,\al}(\ovr{\Om})$ be a solution of \eqref{serrin1}.
\par
There exist two positive numbers $\ep$, $C$, that depend on 
\begin{eqnarray*}
\ds N, f, d_\Om, r_\Om, \mbox{the $C^{2,\al}$-regularity of } \Om, \max_{\ovr{\Om}}u, \min_{\pa\Om} u_\nu, 
\end{eqnarray*}
and a number $\tau\in (0,1)$ such that, if $[u_\nu]_{\pa \Om} \leq \ep$,
then 
\begin{equation*}
B_{r_i} \subset \Om \subset B_{r_e},
\end{equation*}
where $B_{r_i}$ and $B_{r_e}$ are two concentric balls and their radii satisfy 
\begin{equation} 
\label{stability-improved}
r_e-r_i \leq C\,[u_\nu]_{\pa \Om}^{\tau}.
\end{equation}
\end{theorem}
 The number $\tau$ can be explicitly determined (see Theorem \ref{stability serrin}, for details). 
Notice that the semi-norm \eqref{seminorm normal derivative} can be bounded in terms of the deviation used in \eqref{stability ABR}, if we choose $d$ as the minimum of $u_\nu$ on $\partial\Om$.
\par
Therefore, \eqref{stability-improved} significantly improves the estimates of \cite{ABR} and \cite{BNST2}, since it enhances the stability from logarithmic, as in \eqref{stability ABR}, to that of H\"older type proved in \cite{BNST2}, but for {\it any locally Lipschitz non-linearity $f$} (that is, not only for the case $f\equiv 1$).
\par
The assumption on the convexity of $\Om$ can be slightly relaxed (see Theorem \ref{stability serrin}), by requiring that $\Om$ be what we call a $(\cC,\th)${\it -domain}. Roughly speaking, we require that every maximal cap that comes about in employing the method of moving planes has a boundary with a Lipschitz constant bounded uniformly with respect to the direction of mirror reflection chosen (see Section \ref{sec:MovPlanes} for details). 
\par
As already mentioned, our results and the ones in \cite{BNST2} are obtained by using very different techniques. The starting point of \cite{BNST2} is the proof of symmetry given by the same authors in \cite{BNST} which makes use of information, such as Newton inequalities and Pohozaev identity, which is in a sense more global and avoids an extensive use of maximum principles needed to employ the method of moving planes. On the other hand, the method in \cite{BNST2} is more restrictive, since it seems to be suitable only for $f$ constant. 
\par
Instead, based on Serrin's original proof,  our approach is more flexible and allows to treat a {\it general 
Lipschitz continuous non-linearity $f$}; in fact, our study relies on quantitative versions of the maximum principle, such as (global) Harnack-type inequalities. Thus, the quality of the relevant stability estimate is affected by that of the Harnack inequality we employ. This is the reason why, at this stage, even if we can consider a general non-linearity, we need to put a restriction on the type of domain under study. In particular, for convex and $(\cC,\th)${-domains}, we can use improved Harnack-type inequalities in every cap which is generated by the method of moving planes.
\par
Besides by those in \cite{ABR}, the techniques used in this paper are inspired by those employed in \cite{CMS2} (see also \cite{CMS}), where a quantitative study of the radially symmetric configuration was carried out for a related problem --- the {\it parallel surface problem} --- motivated by the remark, made in \cite{MSaihp} (see also \cite{MS1}, \cite{MSm2as}), that {\it time-invariant level surfaces} of solutions of certain nonlinear non-degenerate {\it fast diffusion} equations are parallel surfaces.
\par
As in \cite{ABR} and \cite{CMS2}, our approach consists in fixing a direction, defining an approximate set $X(\de)$ --- built upon the so-called maximal cap (see Section \ref{sec:MovPlanes}) and its mirror-symmetric image in that direction --- which fits $\Om$ well, as the parameter $\de$ tends to $0$. This approximation process is controlled in terms of $\pusn$ and does not depend on the  particular direction chosen. 
\par
The application of Harnack's inequality and Carleson estimates in the maximal cap plays a crucial role in obtaining the new stability estimates. Since we are assuming that every maximal cap has Lipschitz regularity, the improvement in \eqref{stability-improved} is obtained by a refinement of Harnack's inequality in suitable cones.
\par
The paper is organized as follows. In Section \ref{sec:MovPlanes}, we recall the notations
and preliminaries necessary to use the method of moving planes. There, we also discuss on  
the definition of $(\cC,\th)$-domains and present our main result, Theorem \ref{stability serrin}, 
of which Theorem \ref{th:stability-convex} is a straightforward corollary.
In Section \ref{sec:enhanced harnack} --- the core of the paper --- we begin the proof of Theorem \ref{stability serrin}, by producing the necessary enhanced Harnack estimates. Finally, in Section \ref{sec:stability}, we complete the proof of \eqref{stability-improved}.

\setcounter{equation}{0}

\section{Remarks on the method of moving planes and\\ statement of the main result} 
\label{sec:MovPlanes}
We consider a bounded domain $\Om$ of class $C^{2,\alpha}$, $0<\alpha \leq 1$. 
In what follows we will often use the notation $C(N, \Om, d_\Om, r_\Om, f, \dots)$ to denote a constant that
depends on the relevant parameters; in particular, the dependence
on $\Om$ is meant to be only on the $C^{2,\al}$ regularity of $\pa\Om$, as explained in \cite[Remark 1]{ABR}.
The assumed regularity of $\Om$ implies that $\Om$ satisfies a {\it uniform interior sphere condition}; for $x \in \pa \Om$, we let $r(x)$ be the radius of the maximal ball $B \subset \Om$ with $x \in \pa B$ and set
$$
r_\Om = \min_{x \in \pa \Om} r(x).
$$ 
\par
In the sequel, it will be useful to consider the {\it parallel set} of $\Om$:
\begin{equation} \label{Om de}
\Om(\de) = \{x \in \Om:\ \dist(x,\pa \Om)>\de\}.
\end{equation}
If $0<\de<r_\Om$, $\pa\Om(\de)$ is $C^{2,\alpha}$-smooth; also, $\Om(\de)$ satisfies a uniform interior sphere condition with optimal radius $r_\Om-\de$.
 \par
For a unit vector $\om\in\RN$ and a parameter $\mu\in\RR$, we define the following objects:
\begin{equation}
\label{definitions}
\begin{array}{lll}
&\pi_{\mu}=\{ x\in\RN: x\cdot\om=\mu\}\ &\mbox{a hyperplane orthogonal to $\om,$}\\
& \mathcal{H}_\mu=\{ x\in\RN: x\cdot\om>\mu\}\ &\mbox{the half-space \emph{on the right} of $\pi_\mu$,}\\
&\Om_{\mu}=\{x\in A: x\cdot\om>\mu\}\ &\mbox{the right-hand cap of $\Om$},\\
&x^{\mu}=x-2(x\cdot\om-\mu)\,\om\ &\mbox{the reflected image of $x$ in $\pi_{\mu},$}\\
&\Om^{\mu}=\{x\in\RN:x^{\mu}\in \Om_{\mu}\}\ &\mbox{the reflected cap in $\pi_{\mu}$.}
\end{array}
\end{equation}
Set $\La=\sup\{x\cdot\om: x\in \Om\}$, the {\it extent} of $\Om$ in direction $\om$; if $\mu<\La$ is close to $\La$, the reflected cap $\Om^\mu$ is contained in $\Om$ (see \cite{Fr}). Set
\begin{equation}
\label{lambda def}
\la=\inf\{\mu: \Om^{\mu'}\subset \Om \mbox{ for all } \mu'\in(\mu,\La)\}.
\end{equation}
Then at least one of the following cases occurs (see \cite{Se}, \cite{Fr}):
 \begin{enumerate}
\item[(S1)]
$\Om^{\la}$ becomes tangent to $\pa \Om$ at some
point $P^\la\in\pa \Om^\lam \setminus\pi_{\la}$, that is the reflected image of a point $P \in \pa \Om_\lam \setminus\pi_{\la}$;
\item[(S2)]
$\pi_{\la}$ is orthogonal to $\pa \Om$ at some point $Q\in\pa \Om\cap\pi_{\la}.$
\end{enumerate}
The cap $\Om_\la$ will be called the {\it maximal cap}.
\par
As customary in the method of moving planes, we define the function
\begin{equation}\label{w def}
w(x)= u(x^\la)-u(x),\quad x\in \Om_\la;
\end{equation}
$w$ satisfies the equation
\begin{equation} 
\label{eq w semilinear}
\De w+c(x)\,w=0 \  \mbox{ in } \Om_\la,
\end{equation}
where for $x\in\Om_\la$
\begin{equation*}
\label{defc}
c(x)=\left\{
\begin{array}{lll}
\ds\frac{f(u(x^\la))-f(u(x))}{u(x^\la)-u(x)} &\mbox{ if } u(x^\la)\not= u(x),\\
\ds 0 &\mbox{ if } u(x^\la) = u(x).
\end{array}
\right.
\end{equation*}
Notice that $c(x)$ is bounded by the {\it Lipschitz constant} $\frak{L}_f$ of $f$ in the interval
$[0,\max\limits_{\ovr{\Om}}u].$
\par
All the improved estimates in Section \ref{sec:enhanced harnack} concern $w$. As proved in \cite{Se} and refined in \cite{BNV} (see also \cite{Fr}), since $w\ge 0$ on $\pa\Om_\la$, we can assume that $w\ge 0$ in $\Om_\la$. Hence, a standard application of the strong maximum principle to the inequality $\De w-c^-(x)\, w\le 0$ with $c^-(x)=\max[-c(x),0]$ shows that either $w=0$ in $\Om_\la$ (and $\Om$ and $u$ are symmetric about $\pi_\la$) or 
$$
w>0 \ \mbox{ in } \ \Om_\la.
$$

The following lemma ensures that the maximal cap always contains a half ball tangent to $\pa\Om$
at either point $P$ or $Q$.

\begin{lm} \label{lemma ball ABR}
Let $P$ and $Q$ be as in case (S1) and (S2), respectively. Let $B_\rho(p)\subset\Om$ be a ball with
$0<\rho\le r_\Om$ and such that $P\in\pa B_\rho(p)$ or $Q\in\pa B_\rho(p)$.
\par
Then, $p\in \ovr{\Om}_\la$ and $B_\rho(p) \cap \mathcal{H}_\la \subset \Om_\la$.
\end{lm}
\begin{proof}
The assertion is trivial for case (S2). 
\par
If case (S1) occurs, without loss of generality, we can assume that $\om=e_1=(1, 0,\dots, 0)$ and $\la=0$. Since (S1) holds, the point $P^\la$ lies on $\pa\Om$ and cannot fall inside $B_\rho(p)$, since $P\in \pa B_\rho (p)$ and $B_\rho (p)\subset\Om$. Thus,  $|p-P^\la|\geq \rho=|p-P|$ and hence 
$|p_1+P_1| \geq |p_1-P_1|$, which implies that $p_1 \geq 0$, being $P_1>0$.
\end{proof}

As mentioned in the Introduction, the convexity assumption of Theorem \ref{th:stability-convex} can be relaxed; with this purpose, we introduce the class of $(\mathcal{C}, \theta)$-domains that, roughly speaking, have the property that every maximal cap $\Om_\la$ has a boundary with a Lipschitz constant bounded by a number which does not
depend on the direction of reflection chosen.
To this aim, for a fixed $0<t<1/2$, by formula \eqref{Om de}, we define the set
\begin{equation}\label{G def}
G= \Om(t\,r_\Om);
\end{equation}
we know that $G$ is connected (see \cite[pp. 923--924]{ABR}).
Thus, we say that $\Om$ is a {\it $(\mathcal{C}, \theta)$-domain} if 
there exists $\theta > 0$ such that for any direction $\om$ and for any $x\in \Om_\la\setminus\ovr{G}_\la$ 
there exists $\xi \in \pa G_\la \setminus \pi_\la$ such that $x$ and $\xi$ belong to the axis of a (finite) right spherical cone $\mathcal{C} \subset \Om_\la$ with vertex at $x$ and aperture $2\theta$.
\par
This property is not easy to check, since it relies on the knowledge of the position of the critical hyperplane (see Fig.\ref{fig1}). However, the class of $(\mathcal{C}, \theta)$-domains is not empty, as shown by the following proposition.

\begin{figure}[h]
\centering
 \subfigure[A $(\mathcal{C}, \theta)$-domain.]{\includegraphics[width=0.5\textwidth]{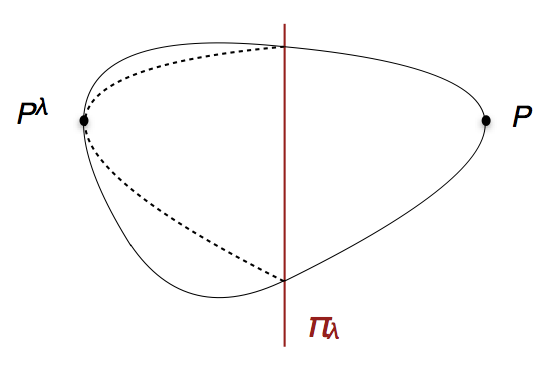}}
 \hspace{2em}
 \subfigure[Not a $(\mathcal{C}, \theta)$-domain.]{\includegraphics[width=0.4\textwidth]{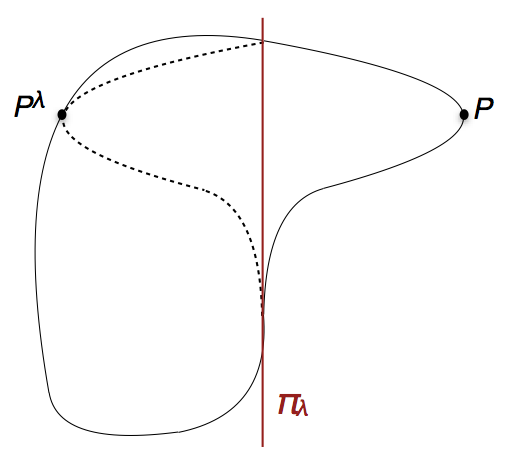}}
  \caption{Every maximal cap of a $(\mathcal{C}, \theta)$-domain has Lipschitz boundary.} \label{fig1}
\end{figure}

%\begin{figure}[h]
%\centering
% \includegraphics[width=0.5\textwidth]{ok_Cth.png}
% \hspace{2em}
% \includegraphics[width=0.4\textwidth]{no_Cth.png}
%  \caption{The domain in the figure on the left satisfies the $(\mathcal{C}, \theta)$ condition. In the figure on the right there is an example of domain which is not a $(\mathcal{C}, \theta)$-domain.} \label{fig1}
%\end{figure}

\begin{prop} 
\label{proposition cone}
Any convex domain $\Om$ of class $C^1$ and satisfying the uniform interior sphere condition with
minimal radius $r_\Om$ is a $(\mathcal{C}, \theta)$-domain with
\begin{equation}
\label{theta}
\theta\ge \arctan \frac{(1-t)\,r_\Om}{2\,d_\Om},
\end{equation}
where $t$ is the parameter appearing in \eqref{G def}.
\end{prop}

\begin{proof}
Since $\Om$ is of class $C^1$, then the method of moving planes can be applied (see \cite{Fr}). As already observed, $G$ satisfies a uniform interior sphere condition with optimal radius $(1-t)\,r_\Om$; also, 
\begin{equation*}
\Om=\{ x+y: x\in G,\, |y|<t\, r_\Om\},
\end{equation*}
and $G$, $G_\la$ and $\Om_\la$ are all convex, since $\Om$ is convex.

From Lemma \ref{lemma ball ABR} we have that $\Om_\la$ contains a half-ball of radius $r_\Om$ with center $p\in \overline{\Om}_\la$, and $G_\la$ contains the half-ball of radius $(1-t)\,r_\Om$ centered at $p$. The maximal ball $B$ contained in this half ball has radius $\ovr{r}=(1-t)\,r_\Om /2$ and is contained in $G_\la$; we denote its center by $y$.
\par
Let $x$ be any point in $\Om_\la \setminus \ovr{G}_\la$ and let $\mathcal{C}$ be the (finite) right circular cone with vertex at $x$ and based on the $(N-1)$-dimensional ball of radius $\ovr{r}$ obtained by intersecting $B$ with the hyperplane perpendicular to the vector $x-y$ and passing through $y$. Since $\Om_\la$ is convex, then $\mathcal{C} \subset \Om_\la$. Moreover, being $B\subset G_\la$, the axis
of $\mathcal{C}$ intersects $\pa G_\la$ at a point $\zi \notin\pi_\la$.
\par
It is clear that the aperture $2\theta$ of $\mathcal{C}$ is such that
$
\theta \geq  \arctan(\ovr{r}/d_\Om),
$
and hence \eqref{theta} holds.
\end{proof}

\begin{corollary}
Let $\Om$ be a convex domain and $\pa \Om$ be of class $C^1$. 
\par
Then any maximal cap $\Om_\la$  has a Lipschitz continuous boundary  $\pa\Om_\la$ with Lipschitz constant bounded by $2d_\Om / r_\Om$.
\end{corollary}
\par
We now state our main result.
\begin{theorem} 
\label{stability serrin}
Let $\Om \subset \RR^N$ be a $(\mathcal{C}, \th)$-domain with boundary of class $C^{2,\al}$
and let $f$ be a locally Lipschitz continuous function with $f(0)\ge 0$.  
\par
There exist two positive constants $\ep$ and $C$, that depend on 
$$
N, \, f, \, d_\Om, \, r_\Om, \, \mbox{the $C^{2,\al}$-regularity of } \Om, \, \max_{\ovr{\Om}} u, \, \min_{\pa\Om} u_\nu,
$$
and a number $\tau\in (0,1)$  such that, if $[u_\nu]_{\pa \Om} \leq \ep$,
then
\begin{equation*}
B_{r_i} \subset \Om \subset B_{r_e},
\end{equation*}
where $B_{r_i}$ and $B_{r_e}$ are two concentric balls and their radii satisfy \eqref{stability-improved}:
\begin{equation*} 
r_e-r_i \leq C\,[u_\nu]_{\pa \Om}^\tau.
\end{equation*}
\end{theorem}
\begin{remark}
(a) As it will be clear from the proof, it holds that $\tau=1/(1+\ga)$ where $\ga$ is given by \eqref{gamma}; $\ga$ depends on the half-aperture $\th$ of the cone $\cC$ in the definition 
of $(\cC,\th)$-domain and on the Harnack's constant for equation \eqref{eq w semilinear}.
\par
(b) It is clear that Theorem \ref{th:stability-convex} is an easy corollary of this theorem, by Proposition \ref{proposition cone}.
\par
(c) As explained in \cite{ABR}, in Theorem \ref{stability serrin} some precautions are in order. 
The dependence of $\ep$ and $C$ on a lower bound for $u_\nu$ is needed in case $f(0)=0$.
In fact, if $f(0)>0$ a comparison argument in an interior ball touching $\pa\Om$ shows that the minimum of $u_\nu$ on $\pa \Om$ can be bounded from below by $f(0)$ times a constant that only depends on $N, \Om, \| u\|_\infty$, and $f$. 
When $f(0)=0$, instead, the first inequality in \eqref{estimate-distance} below does not hold and the constant $\ep$ (and hence $C$) must depend on a lower bound for $u_\nu$. This fact can be seen by considering any (positive) multiple of the first Dirichlet eigenfunction $\phi_1$ for $-\De$: in fact, for any $n\in\NN$ the function $\phi_1/n$ satisfies \eqref{serrin1} with $f(u)=\la_1\,u$,
being $\la_1$ the first Dirichlet eigenvalue;
although $(\phi_1/n)_\nu\to 0$ on $\pa\Om$ as $n\to\infty$, one cannot expect to derive any
information on the shape of $\Om$.
\par
The question whether  Theorem \ref{stability serrin} or \cite[Theorem 1]{ABR} still hold if $f(0)<0$
remains open (see \cite[Section 5.1]{ABR}).
\end{remark}
\par
The proof of Theorem \ref{stability serrin} will be carried out in Sections \ref{sec:enhanced harnack} and 
\ref{sec:stability}.  Since it is quite long and complex, for the reader's convenience, we give an outline which consists of three main steps.
\begin{itemize}
\item[(I)]
%[Step 1.] 
We first improve \cite[Proposition 1]{ABR}. We fix a direction $\om$ and apply the procedure of moving planes. For a sufficiently small $\de>0$, we consider the set $\Om(\de)$ in \eqref{Om de}
and show that there exists a connected component $\Sigma_\de$ of $\Om(\de) \cap \mathcal{H}_\lam$ such that
$$
\|w\|_{L^\infty(\Sigma_\de)} \leq \frac{C}{\de^\ga} \pusn\, ,
$$ 
where $C$ and $\ga$ do not depend on $\de$ and $\om$ (see Proposition \ref{prop:new-prop-1} below).

\item[(II)]
%[Step 2.] 
Let $X(\de)$ be the union of $\Sigma_\de$ and its reflected image in the critical hyperplane $\pi_\lam$. Since $u$ grows linearly near $\pa \Om$, the smallness of $w$ in $\Sigma_\de$ implies that $X(\de)$ fits $\Om$ well (see Lemma \ref{lemma 34 ABR} and Proposition \ref{prop 5 ABR} below). This step gives the approximate symmetry of $\Om$ in the direction $\om$.

\item[(III)]
%[Step 3.] 
Since the arguments in (I) and (II) do not depend on the chosen direction, we apply them for $N$ (mutually) orthogonal directions and obtain a point $\mathcal{O}$ as the intersection  of the corresponding $N$ critical hyperplanes. The point $\mathcal{O}$ can be chosen as an approximate center of symmetry. In fact, we are able to define two balls centered at $\mathcal{O}$ such that Theorem \ref{stability serrin} holds (see the proof of  Theorem \ref{stability serrin} in Section \ref{sec:stability} below). 
\end{itemize}

We notice that once one proves the first step, (I), the remaining ones, (II) and (III), follow from a well-established argument as in \cite{ABR} and \cite{CMS2}. 

\setcounter{equation}{0}

\section{Harnack's inequality in a cone}
\label{sec:enhanced harnack}

In this section, we prove some results based on Harnack's type inequality in a cone, which will be used in Section \ref{sec:stability} to prove step (I) (in particular Proposition \ref{prop:new-prop-1}). 

Let $0<a<1$ be fixed. It is well-known (see \cite[Theorem 8.20]{GT}) that a solution $w$ of \eqref{eq w semilinear} satisfies the following {\it Harnack's inequality}
\begin{equation} \label{harnack}
\sup_{B_{ar}} w \leq \frak{H}_a \inf_{B_{ar}} w ,
\end{equation}
for any ball $B_r \subset \Om_\la$; the Harnack constant $\frak{H}_a$ can be bounded by the power 
$\sqrt{N} + \sqrt{r\,\| c\|_\infty}$ of a constant
%$C_N(a)^{\sqrt{N} + \sqrt{L R}}$, where $C_N(a)$ is a constant
only depending on $N$ and $a$ (see \cite{GT}). For instance, if $c(x) \equiv 0$, by the explicit Poisson's representation formula for harmonic functions, we have that
\begin{equation}\label{Harnack Harmonic}
\sup_{B_{ar}} w \leq \left( \frac{1+a}{1-a} \right)^N \inf_{B_{ar}} w ,
\end{equation}
for any $B_r \subset \Om_\la$  (see \cite{GT}).

The following Lemma consists of an application of Harnack's inequality to a Harnack's chain of  balls 
contained in a cone. The result is well-known 
 (see \cite{Mo}, \cite{JK}, \cite{Ke}); however, since we are interested in a quantitative version of it, we provide our proof.

\begin{lm} 
\label{lemma harnack cone}
Fix a number $a\in (0,1)$. Pick $z \in \pa \Om \cap \ovr{\cH}_\la$ and let $\cC$ be any right spherical cone contained in $\Om_\la$  and with vertex at $z$; let $2\theta$ be the aperture of $\cC$.
\par
Let $w$ be given by \eqref{w def} and pick any two points $x$ and $\xi$ on the axis of $\cC$ such that
$$
|x-z|<|\xi-z|.
$$
\par
Then we have that
\begin{equation}\label{Harnack x e y}
\frac{|x-z|^\ga}{K}\, w (x) \leq w(\xi)\, \leq \frac{K}{|x-z|^\ga}\,w (x),
\end{equation}
where
\begin{eqnarray}
&& K=\frak{H}_a \,\left[\frac{|\xi-z|(1-a\sin\theta)}{1-a}\right]^\ga,\label{K} \\
&& \gamma = \log_\be \frak{H}_a,  \ \ \be= \frac{1+a\sin\theta}{1-a\sin\theta},\label{gamma}
\end{eqnarray}
and $\frak{H}_a$ is given by \eqref{harnack}.
\end{lm}

\begin{proof}
Here we prove the second inequality in \eqref{Harnack x e y}; the first one can be proved similarly.
\par
Let $\ell$ be the unit vector defining the axis of $\cC$ through $z$, that is for instance
\begin{equation*}
    \ell=\frac{x-z}{|x-z|}.
\end{equation*}
We now construct a chain of balls $B_{r_i}(p_i)$, $i=0, 1, \dots, n$, joining $x$ to $y$ with the following
specifications:
\begin{itemize}
\item[(i)] the centers $p_0, p_1,\dots, p_n$ belong to the axis of $\cC$;
\item[(ii)] $p_0=x$, $r_0=|x-z|$ and $\xi\in B_{r_n}(p_n)$;
\item[(iii)] the balls $B_{r_1}(p_1), \dots, B_{r_n}(p_n)$ are all contained in $\cC$ and tangent
to the lateral surface of $\cC$;
\item[(iv)] the radii $r_1,\dots, r_n$ are chosen such that the balls $B_{ar_0}(p_0), \dots, B_{ar_n}(p_n)$ are pairwise disjoint and
\begin{equation*}
\ovr{B_{ar_i}(p_i)} \cap \ovr{B_{ar_{i+1}}(p_{i+1})} = \{p_i + a r_i \ell \},\ i=0,\ldots,n-1.
\end{equation*}
\end{itemize}
\par
A calculation shows that (i)-(iv) determine the $r_i$'s and $p_i$'s as
\begin{equation}
\label{ri&pi}
r_i=r_0\,\frac{(1-a) \sin\th}{1-a\sin\th}\,\be^i, \quad p_i=z+r_0\,\frac{1-a}{1-a\sin\th}\,\be^i\,\ell , \ \ i=1, \dots, n.
\end{equation}
\par
Since $|p_{n-1}-z|\le |\xi-z|$,  from the second formula in \eqref{ri&pi} we obtain a bound for $n$:
\begin{equation}
\label{bound-n}
n\le 1+\frac{\log_{\frak{H}_a}\left(\frac{|\xi-z|}{r_0}\frac{1-a\sin\th}{1-a}\right)}{\log_{\frak{H}_a}\be}
\end{equation}
\par
As usual, the application of Harnack's inequality \eqref{harnack} to each ball $B_{r_i}(p_i)$ of the chain gives that
$
w(\xi)\le\frak{H}_a^n\,w(x);
$
\eqref{Harnack x e y} then follows from \eqref{bound-n} and some algebraic manipulations.
\end{proof}

\begin{remark}
We notice that when $c(x) \equiv 0$ (and hence $\Delta w=0$) we have that
\begin{equation}
\label{gamma-torsion}
\gamma= N \frac{\log \frac{1+a}{1-a}}{\log\frac{1+a \sin\th}{1-a \sin\th}}.
\end{equation}
We observe that $\ga \geq N$ and equality holds only if $\th=\pi/2$.
\end{remark}

\setcounter{equation}{0}

\section{Stability for Serrin's problem: the proof}
\label{sec:stability}

In this section, thanks to the preparatory lemmas of Section  \ref{sec:enhanced harnack} and by following the outline announced in Section \ref{sec:MovPlanes}, we shall bring the proof of Theorem \ref{stability serrin} to an end.

%\begin{theorem} 
%\label{stability serrin}
%Let $\Om \subset \RR^N$ be a $(\mathcal{C}, \th)$-domain with boundary of class $C^{2,\al}$
%and let $f$ be a locally Lipschitz continuous function with $f(0)\ge 0$.  
%Suppose that the solution $u$ of \eqref{serrin1} is such that $u_\nu\ge d$ on $\pa\Om$, for some
%positive constant $d$.
%\par
%There exist two positive constants $\ep$ and $C$, that depend on $N$, $\Om$, $d_\Om$, $\|u\|_\infty$, $f$, and $d$ such that, if $[u_\nu]_{\pa \Om} \leq \ep$,
%then there are two concentric balls $B_{r_i}$ and $B_{r_e}$ such that
%\begin{equation*}
%B_{r_i} \subset \Om \subset B_{r_e}
%\end{equation*}
%and
%\begin{equation} 
%\label{diff radii}
%r_e-r_i \leq C\,[u_\nu]_{\pa \Om}^{\frac{1}{\ga +1}},
%\end{equation}
%where $\ga$ is given by \eqref{gamma}. 
%\end{theorem}
\par
As already mentioned, the crucial step in the proof is item (I) in the outline; it is the analog of \cite[Proposition 1]{ABR}. We modify and improve the procedure used in \cite{ABR} at two salient points: (i) the definition of the approximating symmetric set $X(\de)$; (ii)
the bound on the smallness of the function $w$ in \eqref{w def} in terms of the parameter $\de$ and the semi-norm $\pusn$. 
We draw the reader's attention on the fact that, while (i) is due to a refinement,
based on Carleson-type estimates, of the technique used in \cite{ABR}, for (ii) the assumption that $\Om$ is $(\mathcal{C}, \th)$-domain is necessary --- and, so far, we are not able to remove it --- to treat the case of a general bounded domain with $C^{2,\al}$-smooth boundary.
More precisely, to every maximal cap that comes about in the method of moving planes, we want to apply Lemma \ref{lemma harnack cone} that, roughly speaking, requires that such maximal cap is Lipschitz continuous, which is essentially our definition of $(\mathcal{C}, \th)$-domain, that is certainly fulfilled if $\Om$ is convex, as shown in Proposition \ref{proposition cone}.  In other words, we are so far unable to exclude (even generically) that the situation depicted in Figure 1 (b) may occur.
\par
We now start the procedure.
For a fixed direction $\om$, we let $\la$ be the number defined in \eqref{lambda def}; we then consider the connected component $\Sigma$ of $\Om_\la$ which intersects the interior touching ball at the point $P$, if case (S1) occurs, or at the point $Q$, if (S2) occurs.
\par
For $\de >0$, we consider the set $\Om(\de)$ defined in \eqref{Om de} and
define $\Sigma_\de$ as the connected component of
$
\Om(\de) \cap \mathcal{H_\la}
$
contained in $\Sigma$ --- notice that our definition of $\Sigma_\de$ differs from that in \cite{ABR}.
In addition, we fix the domain $G$ in \eqref{G def} by choosing $t=1/32$ and we set $a=1/2$ in Lemma \ref{lemma harnack cone}.

\begin{prop}
\label{prop:new-prop-1}
Let $\Om$ be a $(\mathcal{C}, \th)$-domain. 
\par
Then, there is a constant $C=C(N, \Om, d_\Om, r_\Om, f, \max_{\ovr{\Om}} u)$ such that
\begin{equation}
\label{w Linfty Sigma}
\|w\|_{L^\infty(\Sigma_\de)} \leq C\, \de^{-\ga} [u_\nu]_{\pa \Om}
\ \mbox{ for } \ 0<\de \le r_\Om/32, 
\end{equation}
where
$\ga$ and $\beta$ are the numbers defined in \eqref{gamma}.
\end{prop}

\begin{proof}
Let $\widetilde{\Si}=\{x\in \Sigma:\ \dist(x,\pa \Sigma)> r_\Om/64\}$. We apply  \cite[Proposition 1]{ABR} with $\de=r_\Om/64$ and obtain that 
\begin{equation*}
\|w\|_{L^\infty(\widetilde{\Si}_\de)} \leq \widetilde{C}\,
[u_\nu]_{\pa \Om},
\end{equation*}
for some constant $\widetilde{C}=\widetilde{C}(N, \Om, d_\Om, r_\Om, f, \| u\|_\infty)$ . 
By using an argument analogous to the ones in the proofs of \cite[Lemmas 3.2 and 4.2]{CMS2} -- which are based on boundary Harnack's inequality -- we can extend the previous estimate to $G_\la$:
\begin{equation} \label{w Gm serrin}
\|w\|_{L^\infty(G_\la)} \leq M C\, [u_\nu]_{\pa \Om},
\end{equation}
where $M$ is the constant appearing in the boundary Harnack inequality (see Theorem 1.3 in \cite{BCN}). It is important to notice that the bound in \eqref{w Gm serrin} does not depend on $\de$.

Let $x$ be any point in $\Sigma_\de \setminus \overline{G}_\la$, with $\de < r_\Om/32$. Since $\Om$ is a $(\mathcal{C}, \th)$-domain, from Lemma \ref{lemma harnack cone}, by choosing $\xi$ equal to the point $\zi\in \pa G_\la\setminus\pi_\la$ (in the definition of $(\mathcal{C}, \th)$-domain)  
in Lemma \eqref{lemma harnack cone}, we obtain that
\begin{equation*}
    w(x) \leq K\, \de^{-\ga} w(\zi);
\end{equation*}
\eqref{w Gm serrin} then yields \eqref{w Linfty Sigma} with $C=\widetilde{C} KM$.
\end{proof}

\begin{remark}
The statement of Proposition \ref{prop:new-prop-1} differs from that of \cite[Proposition 1]{ABR}
in three aspects: (i) the set $\Si_\de$ we consider extends up to the hyperplane $\cH_\la$; (ii) the dependence on $\de$ in \eqref{w Linfty Sigma} greatly improves that of \cite[Eq. (8)]{ABR}, that blows up exponentially as $\de\to 0$; (iii) the seminorm $[u_\nu]_{\pa\Om}$ replaces the norm
$\| u_\nu-d\|_{C^1(\pa\Om)}$. 
\end{remark}

Now, we define a symmetric open set as
\begin{equation}
\label{X def}
X(\de)=\mbox{the interior of} \ \,\Si_\de \cup \Si_\de^{\la} \cup (\pa \Si_\de \cap \pi_\la);
\end{equation}
we want to show that $X(\de)$ \emph{fits $\Om$ well}; how well is the main point of this paper.
\par
The main idea is to combine the estimate \eqref{w Linfty Sigma} on the smallness of $w$, together with the fact the $u$ grows linearly near $\pa \Om$; as shown in  \cite[Proposition 4]{ABR} indeed, we know that
\begin{equation}
\label{estimate-distance}
\underline{K}\,\dist(x,\pa\Om)\le u(x)\le \overline{K}\,\dist(x,\pa\Om) \ \mbox{ for all} \ x\in\Om,
\end{equation}
where 
$$
1/\underline{K}, \ovr{K}\le C=C(\Om,d_\Om, \max_{\ovr{\Om}}u,f,\min_{\pa\Om} u_\nu).
$$ 
In the following lemma we give our version  \cite[Eq.(34)]{ABR}.

\begin{lm} \label{lemma 34 ABR}
For $0<\si, \de \leq r_\Om/16$ let $\Om(\si)$ and $X(\de)$ be the sets defined in \eqref{Om de}
and \eqref{X def}, respectively.  Let $\ga$ and $C$ be given by \eqref{gamma} and \eqref{w Linfty Sigma}, respectively. 
\par
If 
\begin{equation} \label{K si geq}
\underline{K}\,\si > C \de^{-\ga} [u_\nu]_{\pa \Om} + \ovr{K}\,\de,
\end{equation}
then we have that
\begin{equation}
\label{Om sig subset X de}
\Om(\sigma) \subset X(\de) \subset \Om.
\end{equation}
\end{lm}

\begin{proof}
We have that $X(\de) \subset \Om$ by construction.
\par
To show the first inclusion in \eqref{Om sig subset X de} we proceed by contradiction. Since the maximal cap $\Om_\la$ contains a ball of radius $r_\Om/4$, then $X(\de)$ intersects $\Om(\sigma)$. Assume that there exists a point $y \in \Om(\sigma) \setminus X(\de)$ and let $x$ be any point in $X(\de) \cap \Om(\si)$. Since $\Om(\si)$ is connected, $x$ is joined to $y$ by a path contained in $\Om(\si)$. Let $z$ be the first point on this path which falls outside $X(\de)$. It is clear that $z \in \pa X(\de) \cap \Om(\si)$. We now consider two cases.
\par
If $z\cdot \omega < \lam$ then the reflection $z^\lam$ of $z$ in $\pi_\lam$ is such that $z^\lam \in \pa \Sigma_\de$ and $\dist(z^\lam,\pa \Om) =\de$. Since $u(z)=w(z^\lam)+u(z^\lam)$, from \eqref{w Linfty Sigma} and  \eqref{estimate-distance} we have that
\begin{equation} \label{34 by contrad}
\underline{K}\,\si\le u(z) \leq C \de^{-\ga} [u_\nu]_{\pa \Om} + \ovr{K}\,\de,
\end{equation}
a contradiction.
\par
If $z\cdot \omega \geq \lam$, then $z \in \pa \Sigma_\de$ and $\dist(z,\pa \Om) =\de$. Hence, from \eqref{estimate-distance} we obtain that $u(z) \leq \ovr{K}\,\de$ and \eqref{34 by contrad} holds as well.
Since $z \in \Om(\si)$ then from \eqref{estimate-distance} we have that $u(z) \geq \underline{K}\, \si$ which contradicts \eqref{K si geq} on account of \eqref{34 by contrad}.
\end{proof}

We draw the reader's attention on the differences in \eqref{K si geq} compared to \cite[Eq. (34)]{ABR}:
(i) thanks to Lemma \ref{lemma harnack cone}, the term $\de^{-\ga}$ replaces one in \cite[Eq. (34)]{ABR} that blows up exponentially; (ii) due to the different definition of the symmetrized set $X(\de)$ (denoted by $X_\de$ in \cite{ABR}), we
simplify the last summand in \eqref{K si geq}.

\begin{prop} 
\label{prop 5 ABR}
Let $\ga$ be given by \eqref{gamma}. 
There exist positive numbers $C$ and $\ep$, depending on $\Om,d_\Om, r_\Om, \max_{\ovr{\Om}}u,f,\min_{\pa\Om} u_\nu$, and $\sigma,\, \de>0$ such that  \eqref{Om sig subset X de} holds with
%\begin{equation*}
%%\label{Om sig subset X de}
%\Om(\sigma) \subset X(\de) \subset \Om,
%\end{equation*}
%and
\begin{equation} 
\label{de leq si}
\de < \si < C\, [u_\nu]_{\pa \Om}^\frac{1}{\ga+1},
\end{equation}
provided that $[u_\nu]_{\pa \Om}\le\ep$.
\end{prop}

\begin{proof}
We must choose $\de, \si\le r_\Om/16$ that satisfy \eqref{K si geq}. 
If we let
\begin{equation*}
\de = \left(\frac{C}{\ovr{K}}\right)^{\frac{1}{\ga+1}} [u_\nu]_{\pa \Om}^{\frac{1}{\ga+1}}
\ \mbox{ and } \ 
\si = \frac{4\ovr{K}}{\underline{K}} \de,
\end{equation*}
since $\ovr{K}\,\de=C\,\de^{-\ga}\,[u_n]_{\pa\Om}$, then \eqref{K si geq} holds provided that
$\si\le r_\Om/16$, that is for
\begin{equation*}
\ep \leq \frac{\ovr{K}}{C} \left( \frac{r_\Om \underline{K}}{64\,\ovr{K}} \right)^{\ga +1},
\end{equation*}
The conclusion then follows from Lemma \ref{lemma 34 ABR}.
\end{proof}
\par
We are now ready to prove the main result of this paper.

\begin{proof}[Proof of Theorem \ref{stability serrin}]
%By using \eqref{de leq si} in place of formula (33) in \cite{ABR}, the proof is analogous to the one of Theorem 1 in \cite{ABR}.
The proof is analogous to that of \cite[Theorem 1]{ABR}; here, we use
\eqref{de leq si} in place of \cite[Eq. (33)]{ABR}. 
\par
Indeed, for any fixed direction $\om\in\SSN$, 
Proposition \ref{prop 5 ABR} implies that, if $\si$ satisfies \eqref{de leq si},  then for any $x \in \pa \Om$ there
exists $y_\om\in \pa \Om$ such that
\begin{equation} 
\label{x lam y}
|x^\lambda - y_\om | \leq 2 \sigma,
\end{equation}
where $x^\la$ is the reflection of $x$ in the critical hyperplane $\pi_\la$ with $\la$ defined by 
\eqref{lambda def} (see also \cite[Corollary 2]{ABR}). 
\par
We now choose $N$
orthogonal directions, say $e_1,\ldots,e_N$, denote by
$\pi_1,\ldots,\pi_N$ the corresponding critical hyperplanes and we place the origin $0$ of $\RN$
at the (unique) point in common to all the $\pi_j$'s. If we denote by
$R_j(x)$ the reflection of a point $x$ in $\pi_j$, we have that 
$$
-x= (\mathcal{R}_{N} \circ \cdots \circ \mathcal{R}_{1})(x);
$$
also, if we set $y_0=x$, $y_1=y_{e_1}$  and, for $j=2,\dots, N$,  $y_j$ as the point in $\pa\Om$ such
that $|\mathcal{R}_j(y_{j-1})-y_j|\le 2 \si$ determined by \eqref{x lam y}, we obtain that
\begin{multline*}
|-x-y_N|\le \\
|(\mathcal{R}_{N} \circ \cdots \circ \mathcal{R}_{1})(x)-\mathcal{R}_{N}(y_{N-1})|+
|\mathcal{R}_{N}(y_{N-1})-y_N|\le \\
|(\mathcal{R}_{N-1} \circ \cdots \circ \mathcal{R}_{1})(x)-y_{N-1}|+2 \si\le \\
|(\mathcal{R}_{N-2} \circ \cdots \circ \mathcal{R}_{1})(x)-y_{N-2}|+4 \si\le \cdots \le 2N\si.
\end{multline*}
\par
Thus, we showed that there exists $y=y_N\in\pa\Om$ such that
\begin{equation*} \label{Rx lam y}
|x +y | \leq 2 N \sigma.
\end{equation*}
This fact and \eqref{x lam y} imply that $0$ is an approximate
center of symmetry for $\Om$, in the sense of \cite[Proposition 6]{ABR}, that is, for any direction $\om$,  for the critical hyperplane $\pi$ in the direction $\om$ we have that
\begin{equation} \label{bound prop 6 ABR}
\dist(0,\pi) \leq 4N\,[1+d_\Om]\,\si.
\end{equation}
\par
By letting
\begin{equation*}
r_i=\min_{x \in \pa \Om} |x|, \quad \textmd{and }\ \ r_e=\max_{x \in \pa
\Om} |x|,
\end{equation*}
we obtain \eqref{stability-improved} thanks to \eqref{bound prop 6 ABR}, \cite[Proposition 7]{ABR} and \eqref{de leq si}.
\end{proof}
%\par
%Thus, Theorem \ref{stability serrin} considerably improves the logarithmic stability obtained in \cite{ABR},
%at the cost of restricting the class of domains taken into account, but still for a general locally
%Lipscitz continuous non-linearity $f$.
\par
Notice that Theorem \ref{stability serrin} was proved by fixing $a=1/2$ in Lemma \ref{lemma harnack cone}. However, an analog of Theorem \ref{stability serrin} for any fixed $a \in (0,1)$ can be proved by the same arguments and the number $\ga$ in \eqref{gamma}  and the exponent in \eqref{stability-improved} may be optimized in terms of $a$. It is clear that the constants $\ep$ and $C$ will change accordingly  (the smaller is $\ga$, the larger is $C$ and smaller is $\ep$). 
By using this plan, in case $\Om$ is convex and $u$ is a solution of the torsional rigidity problem  
--- i.e. when $f(u)=1$ in \eqref{serrin1} --- we are able to give more explicit formulas
and partially compare our results to those in \cite{BNST2}.

\begin{corollary}
Let $\Om$ be a bounded convex domain with boundary of class $C^{2,\al}$. Let $u$ be a solution of 
$$
\Delta u = -1 \ \mbox{ in } \ \Om, \quad u=0 \ \mbox{ on } \ \pa\Om.
$$
\par
Then, for any $\eta>0$ there exist constants $C$ and $\ep$, that
depend on $N, \Om,d_\Om, r_\Om$ and $\eta$,
such that
\begin{equation} \label{stability convex torsion}
r_e-r_i \leq C\,[u_\nu]_{\pa \Om}^{\frac{1}{\tau+\eta}},
\end{equation}
provided that $[u_\nu]_{\pa \Om} \leq \ep$, where
\begin{equation} \label{tau}
\tau = 1+N\,\sqrt{1+\left[\frac{2d_\Om}{r_\Om}\right]^2}.
\end{equation}
\end{corollary}

\begin{proof}
We recall that, in the case in hand, $\ga$ is given by \eqref{gamma-torsion} and is increasing as $a$ 
grows; thus, the optimal exponent should be looked for when $a\to 0^+$.  Therefore, for
a $(\mathcal{C},\th)$-domain, the exponent in \eqref{stability-improved} behaves as 
$$
\frac1{1+\gamma}= \frac{1}{1+N/\sin\th+ o(1)} \ \mbox{ as } \ a\to 0^+.
$$ 
The conclusion then follows from Proposition \ref{proposition cone}.
\par
We finally notice that, in this case, the dependence of the constants $C$ and $\ep$ appearing
in Theorem \ref{stability serrin} on the quantities $\|u\|_\infty$ and $\min_{\pa\Om} u_\nu)$ can be removed, since these quantities can be bounded in terms of the $C^2$ regularity of $\Om$, by standard barrier techniques.
\end{proof}
\begin{remark}
In \cite[Theorem 2]{BNST2} the authors prove an estimate similar to \eqref{stability convex torsion}. When $\Om$ is convex, our estimate improves that in \cite[Theorem 2]{BNST2} if 
$$
\frac{d_\Om}{r_\Om} \leq \sqrt{\frac{2N^2+N-5/2}{N}}.
$$
\end{remark}

\section*{Acknowledgements} 
The authors wish to thank Paolo Salani for the useful discussions we had together.

The authors have been supported by the Gruppo Nazionale per l'Analisi Matematica, la Probabilit\`a e le loro
Applicazioni (GNAMPA) of the Istituto Nazionale di Alta Matematica (INdAM). 

The paper was completed while the author was visiting \lq\lq The Institute for Computational Engineering and Sciences\rq\rq (ICES) of The University of Texas at Austin, and he wishes to thank the Institute for hospitality and support. The author has been also supported by the NSF-DMS Grant 1361122 and the Firb project 2013 ``Geometrical and Qualitative aspects of PDE''.

\end{document}